\documentclass[12 pt,leqno]{article}

\usepackage{amsmath}
\usepackage{amssymb}
\usepackage[english]{babel}
\usepackage{amsthm}
\usepackage[latin1]{inputenc}
\usepackage{hyperref}
\usepackage{graphicx}
\usepackage{subfigure}
\usepackage{makeidx}

\newtheorem{teo}{Theorem}
\newtheorem{lemma}{Lemma}
\newtheorem{prop}{Proposition}
\newtheorem{defin}{Definition}
\newtheorem{cor}{Corollary}
\newtheorem{remark}{Remark}

\newenvironment{sistema}%
{\left\lbrace\begin{array}{@{}l@{}}}%
{\end{array}\right.}

{\left( \begin{array}{@{}l@{}}}%
{\end{array}\right)}

\begin{document}

\title{\textbf{Periodic Solutions of a Forced Relativistic Pendulum via Twist Dynamics}}
\author{\textbf{Stefano Marò} \\
\textit{\small{Dipartimento di Matematica - Università di Torino}}\\
\textit{\small{Via Carlo Alberto 10, 10123 Torino - Italy}}\\
\textit{\small{e-mail: stefano.maro@unito.it}}
}
\date{}

\maketitle

\begin{abstract}
\noindent We prove the existence of at least two geometrically different periodic solution with winding number $N$ for the forced relativistic pendulum. The instability of a solution is also proved. The proof is topological and based on the version of the Poincaré-Birkhoff theorem by Franks. Moreover, with some restriction on the parameters, we prove the existence of twist dynamics.
\end{abstract}

\section{Introduction and statement of the main results}

\bigskip

\noindent We are concerned with the equation of the forced relativistic pendulum
\begin{equation}\label{introeq}
\frac{d}{dt} \Bigl(  \frac{\dot x}{\sqrt{1-{\dot x}^2}}   \Bigr) + a \sin x = f(t),
\end{equation}
where $a$ is a positive real constant and $f$ is a $T$-periodic real function with mean value zero over a period. 

\smallskip

\noindent Results on the existence of solutions to (\ref{introeq}) and on their qualitative properties have been recently given (by different methods) by various authors. 
We refer to the works of Brezis-Mawhin \cite{brezismawhin} where it is proved the existence of a $T$-periodic solution and of Bereanu-Torres \cite{bereanutorres} who provided the existence of a second solution not differing from the previous by a multiple of $2\pi$. From now on, multiplicity has to be understood in this way. These works use very delicate variational techniques. As for topological methods we refer to Bereanu-Jebelean-Mawhin \cite{bereanujebeleanmawhin} where it is proved the existence of two $T$-periodic solutions using degree arguments and upper and lower solutions. In this case, a restriction on the parameters $a$ and $T$ is required.

\smallskip

\noindent Another interesting topic concerning (\ref{introeq}) is the question of the stability of the solutions; in this direction, Chu-Lei-Zhang \cite{chuleizhang} found stable $T$-periodic solutions for the related case of a relativistic pendulum with variable length using KAM theory and Birkhoff normal form. 

\smallskip

\noindent Our contribution in the present paper contains two main results (Theorem \ref{introma} and Theorem \ref{twist} below) which we describe in what follows.

\smallskip

\noindent Our first result deals with the existence of $T$-periodic solutions with winding number $N$, i.e. solutions $x(\cdot)$ such that $x(t +T)= x(t)+2N \pi, N \in \mathbb{Z}$  for all $t \in\mathbb{R}$. Obviously,  $T$-periodic solutions correspond to the case $N = 0$. We first show (Proposition \ref{neces}) that the condition

\begin{equation}\label{main} 
\Bigl|\frac{2N\pi}{T}\Bigr|<1
\end{equation} 
is necessary for the existence of $T$-periodic solutions with winding number $N$; this agrees with physical intuition, which suggests (due to the bound arising from the speed of light) that we cannot expect solutions for every value of $N$ and $T$. The necessary condition is also sufficient; in fact we prove

\begin{teo}\label{introma}
For every $N \in \mathbb{Z}$ such that $|\frac{2N\pi}{T}|<1$ there exist at least two solutions such that $x(t + T) = x(t)+2N \pi, N \in \mathbb{Z}$  for all $t \in \mathbb{R}$.  Moreover, at least one of them is unstable.
\end{teo}

\noindent The crucial fact is the passage from the Lagrangian to the Hamiltonian form performed by the Legendre transform. This eliminates the singularity and allows us to obtain solutions as fixed points of the Poincar\'e map associated to the planar Hamiltonian system corresponding to (\ref{introeq}). If we consider the case $N=0$ we obtain 
\begin{equation}\label{hintr}
\dot{q}=H_p,\quad \dot{p}=-H_q
\end{equation} 
where $H=\sqrt{p^2+1} - a\cos q - f(t)q$. This transform has been used by Mawhin in \cite{mawcime} for other purposes. More precisely, in Section 2 we first give a version of the Poincar\'e-Birkhoff theorem (Theorem \ref{franks}) which is a variant of the result by Franks \cite{franks}. Then, in Section 3 we show (Lemma \ref{l6} and Lemma \ref{l4}) that under (\ref{main}) the assumptions of Theorem \ref{franks} are satisfied. The proof of Theorem 4.2 in \cite{franks} is developed by elaborated techniques from differential geometry, while the proof of our Theorem \ref{franks} requires more elementary techniques based on the application of a result of Le Calvez-Wang \cite{lecalvez} (Theorem \ref{pbstrip}). The crucial notion for our argument is that of exact symplectic map (see \cite{ortegakunze} for an introduction on exact symplectic maps).  It is worth mentioning that Franks gave in \cite{franks} applications of his abstract result to the equation of the forced (non-relativistic) pendulum; anyway they need some clarification. As for a comparison with this and other results on the (non-relativistic) pendulum, we refer the interested reader to the end of this introductory section. For the proof of the instability result in Theorem \ref{introma} we use a theorem by Ortega \cite{ortega1}, together with the Poincaré-Hopf Theorem \cite{granasdugundji}.

\smallskip

\noindent Our second result provides a restriction on the parameter $a$ which makes the Poincaré map a "twist map", whose definition is recalled at the beginning of Section 2.2. More precisely in Section 3.2 we prove, by a Sturm comparison argument,
\begin{teo}\label{introtw}
If $a<\frac{\pi^2}{T^2}$, then the Poincaré map associated to system (\ref{hintr}) is twist. If $a=\pi^2/T^2$ and $f$ does not vanish identically the same result holds true. Moreover the condition $a\leq\frac{\pi^2}{T^2}$ is optimal. 
\end{teo}
\noindent This important fact leads us to more information on the number of solutions and their stability properties. To state these results, we adapt a definition given in \cite{ortega2} saying that a planar first order system in the variables $(q,p)$ is degenerate if there exists a curve $(q_s(0), p_s(0))$ such that the application $s\mapsto q_s(0)$ is defined from $\mathbb{R}$ onto $\mathbb{R}$, satisfies $q_{s+2\pi}(t)=q_s(t)+2\pi$ and $p_{s+2\pi}(t)=p_s(t)$, is bijective in $[0,2\pi)$ and continuous and for every $s\in [0,2\pi)$ the point $(q_s(0),p_s(0))$ is the initial condition of a $T$-periodic solutions with winding number $N$. Then we can state

\begin{teo}\label{twist}
If $0<a< \pi^2 / T^2$ either the number of isolated $T$-periodic solutions with winding number $N$ is finite or we are in the degenerate case and every degenerate solution is unstable. If $a=\pi^2 / T^2$ and $f(t)$ is not the trigonometric function $a \sin( \frac{2N\pi}{T}t)$, the same result holds true.
\end{teo}
\noindent The qualitative properties of the $T$-periodic solutions with winding number $N$ (whose existence is guaranteed by Theorem \ref{introma} and, in the context of twist maps, by Ortega's result \cite{ortega3} recalled in Theorem \ref{ortteo}) follow from the application of two abstract facts for planar, exact symplectic and twist maps (Corollary \ref{cor1} and Corollary \ref{cor2} in Section 2) which follow from Theorem \ref{ortteo} as well. At this stage, it is important to recall and develop the "intersection property". 

\smallskip

\noindent We end this introductory section by a comparison between our results (and their method of proof) and some (somehow analogue) results for the forced non-relativistic pendulum equation
\begin{equation}\label{introcl}
\ddot{x}+a\sin x=f(t)
\end{equation}
\noindent with analogous hypothesis as in equation (\ref{introeq}). It is well known that it has at least two $T$-periodic solutions. The existence of one solution was proved by Hamel \cite{hamel} and rediscovered independently by Dancer \cite{dancer} and Willem \cite{willem}. Then the existence of a second solution has been proved by Mawhin and Willem \cite{mawhinwillem} using mountain pass arguments. It is worth mentioning that also Franks \cite{franks} gave his contribution using his version of the Poincaré-Birkhoff theorem. More details on Franks's result will be given at the end of section 3.1. We stress that our approach and results apply also in the case of equation (\ref{introcl}): it is just a matter of computation, the arguments are completely analogous. Therefore our results are an improvement also concerning the non-relativistic case.

\smallskip

\noindent The paper is organised as follows. In Section 2 we state and prove the abstract results for planar maps that we need in the sequel. More precisely, in Section 2.1 we deal with the existence of two fixed points of a planar exact symplectic map (Theorem \ref{franks}); in Section 2.2 we focus on exact symplectic and twist maps, recall the existence of two fixed points (Theorem \ref{ortteo}) and give some qualitative results on the number of such solutions (Corollary \ref{cor1}) and of their index (Corollary \ref{cor2}). In Section 3 we prove, using the results of Section 2, Theorem \ref{introma}, Theorem \ref{introtw} and Theorem \ref{twist}.
In Section 4 we give some more detailed results in the particular case when equation (\ref{introeq}) is autonomous. This is done for a better understanding of the main results.

\smallskip

\noindent I want to thank Professor Rafael Ortega who introduced me to this problem. Without his constant and patient supervision this work would have not been possible. I am indebted to Professor Anna Capietto whose suggestions have been very useful for the final layout.

\section{Two versions of the Poincaré-Birkhoff Theorem}

\subsection{If the annulus is not invariant}
The classical Poincaré-Birkhoff theorem gives two fixed points of an area preserving homeomorphism of the annulus which satisfies a boundary twist condition. In this section we will state a variation of this theorem due to Franks \cite{franks} in which it is eliminated the invariance of the annulus despite an increasing of regularity.\\
\smallskip
\noindent In a plane with coordinate $(\theta,r)$, consider, for $0<a<b$, the two strips $\tilde{A}=\mathbb{R}\times[-a,a]$ and $\tilde{B}=\mathbb{R}\times[-b,b]$ so that $\tilde{A}\subset \tilde{B}$, and the corresponding annuli $A=\mathbb{S}^1\times[-a,a]$ and $B=\mathbb{S}^1\times[-b,b]$ so that $A\subset B$.\\
We will deal with an embedding $\tilde{S}:\tilde{A}\rightarrow \tilde{B}$ defined as follows:
$$\tilde{S}(\theta,r)=(\theta_1,r_1)$$
and
$$
\begin{sistema}
\theta_1=F(\theta,r)\\
r_1=G(\theta,r)
\end{sistema}
$$
where $F,G:\mathbb{R}^2\rightarrow\mathbb{R}$ are functions of class $C^2$ such that
$$
F(\theta+2\pi,r)=F(\theta,r)+2\pi,\quad G(\theta+2\pi,r)=G(\theta,r).
$$
By the definition of $\tilde{S}$ we are allowed to think $\tilde{S}$ as a lift of a map $S:A\rightarrow B$ defined on the cylinder. Moreover we will suppose that the map $S$ is isotopic to the inclusion $i$, i.e. there exists a $C^2$ map $f:A\times[0,1]\rightarrow B$ such that for every $t\in[0,1]$, $f_t(x)=f(t,x)$ is a diffeomorphism from $A$ onto its image, $f_0=S$ and $f_1=i$.
\smallskip
  
\noindent Now consider the standard volume form $\omega=d\theta\wedge dr$ and remember that $\tilde{S}$ is area preserving if $d\theta\wedge dr=d\theta_1\wedge dr_1$. Now, if we consider the $1$-form
$$
\alpha=r_1d\theta_1-rd\theta,
$$
we have that $\tilde{S}$ is symplectic if and only if $\alpha$ is closed. But the strip is contractible, so closed and exact forms coincide. Summing up we have that $\tilde{S}$ is symplectic if and only if there exists $V\in C^3(\tilde{A})$ such that
$$
dV=r_1d\theta_1-rd\theta.
$$
This equivalence is no longer true in the cylinder, in fact the primitive $V$ could be multi-valued. So we have the following

\begin{defin}
We say that $\tilde{S}$ is \emph{exact symplectic} if there exists a function $V\in C^3(\tilde{A})$ such that
$$
dV=r_1d\theta_1-rd\theta\quad \mbox{and}\quad V(\theta+2\pi,r)=V(\theta,r)\quad \forall (\theta,r).
$$
\end{defin}
\noindent Now we can state the slightly modified theorem by Franks. He dealt with a map defined from an annulus into itself, we will need the case of a map defined from a strip into itself. His proof deals with very sophisticated techniques of differential geometry. We will reach the requested version translating his proof in our concrete case of the cylinder so that it will be understandable also for people who do not deal with those abstract tools.   

\begin{teo}\label{franks}
Consider a map $\tilde{S}:\tilde{A}\rightarrow\tilde{B}$ which is the lift of an exact symplectic embedding $S:A\rightarrow B$ isotopic to the inclusion such that $S(A)\subset int B$. Suppose that the following boundary twist condition is satisfied: there exists $r>0$ such that 
\begin{equation*}
\begin{split}
& F(\theta,a)-\theta>r, \quad \theta\in [0,2\pi)        \\
& F(\theta,-a)-\theta<-r, \quad \theta\in [0,2\pi).
\end{split}
\end{equation*} 
Then $\tilde{S}$ has at least two fixed points.
\end{teo}

\noindent The strategy of the proof is to extend $\tilde{S}$ to an homeomorphism $\tilde{g}$ of the strip $\tilde{B}$. So we can use the fact that the fixed points of the Poincaré-Birkhoff theorem are in fact fixed points of the lift. More precisely we will use the following result in \cite{lecalvez}
\begin{teo}\label{pbstrip}
Let $\tilde{\mathcal{A}}$ be a strip and let $\mathcal{A}$ be its corresponding annulus. Consider a map $\tilde{S}:\tilde{\mathcal{A}}\rightarrow\tilde{\mathcal{A}}$ which is the lift of a homeomorphism $S:\mathcal{A}\rightarrow \mathcal{A}$, isotopic by homeomorphisms to the identity and area preserving. Suppose that the following boundary twist condition is satisfied: 
\begin{equation}\label{bt1}
\begin{split}
& F(\theta,a)-\theta>0, \quad \theta\in [0,2\pi)        \\
& F(\theta,-a)-\theta<0, \quad \theta\in [0,2\pi).
\end{split}
\end{equation} 
Then $\tilde{S}$ has at least two fixed points.
\end{teo}

\noindent At the beginning we will work in the annuli $A$ and $B$. We observe that in this case is not possible to apply the isotopy extension theorem as stated in \cite[p.63]{Milnor} because $B$ should be without boundary. Anyway, by the hypothesis $S(A)\subset int B$ it is possible to slightly modify it in order to extend $S$ to $g_1:B\rightarrow B$ so that $g_1$ is isotopic to the identity and restricted to a neighbourhood of $\partial B$ is the identity (this is achieved in \cite[Theorem 1.3, p.180]{hirsch}). Notice that $g_1$ does not preserve the area out of $A$.\\
Now choose $a_0$ slightly smaller than $a$ and define the following subsets that we will use during the proof  
\begin{equation*}
\begin{split}
& \tilde{A}_0=\mathbb{R}\times [-a_0,a_0]        \\
& A_0=\mathbb{S}^1\times [-a_0,a_0].
\end{split}
\end{equation*} 
In order to apply theorem \ref{pbstrip} and get the result, let us prove the following two lemmas:
\begin{lemma}\label{claim1}
It is possible to alter $g_1:B\rightarrow B$ finding a diffeomorphism $g_2:B\rightarrow B$ such that
\begin{itemize}
\item $g_2$ is area preserving on $B$,
\item $g_2|_{A_0}=S$,
\item $g_2|_{\partial B}=Id$.
\end{itemize}
\end{lemma}

\begin{lemma}\label{claim2}
It is possible to alter the lift $\tilde{g}_2:\tilde{B}\rightarrow\tilde{B}$ finding $\tilde{g}:\tilde{B}\rightarrow\tilde{B}$ such that
\begin{itemize}
\item $\tilde{g}$ is area preserving on $\tilde{B}$,
\item $\tilde{g}$ has no fixed points out of $\tilde{A}$,
\item $\tilde{g}$ satisfy the boundary twist condition (\ref{bt1}),
\item $\tilde{g}=\tilde{g_2}$ on $\tilde{A}$.
\end{itemize}
\end{lemma}
So theorem \ref{franks} will follow from the application of theorem \ref{pbstrip} to $\tilde{g}$.  
\begin{proof}[Proof of Lemma \ref{claim1}]
To prove this lemma we will use Moser's ideas, presented in \cite{moserman} in a more general framework.
Let us break the proof in several steps.\\

\textbf{\textit{Step $1$}}.  Let $B^+=\mathbb{S}^1\times [a_0,b]\subset B$. It results
$$
\mu(B^+)=\int_{B^+}\det g_1^\prime(\theta,r)d\theta dr.
$$
Let $g_1(\theta,r)=(\theta_1,r_1)$. The $1$-form $r_1d\theta_1-rd\theta$ is exact symplectic on $A$ so its integral over any closed path in $A$ must vanish, in particular we have
$$
\int_{\mathbb{S}^1\times \{a_0\}}rd\theta=\int_{g_1(\mathbb{S}^1\times \{a_0\})}rd\theta.
$$
Moreover, because $g_1$ is the identity over $\mathbb{S}^1\times \{b\}$ we have 
$$
\int_{\mathbb{S}^1\times \{b\}}rd\theta=\int_{g_1(\mathbb{S}^1\times \{b\})}rd\theta
$$
so that
$$
\int_{\partial B^+}rd\theta=\int_{\partial B^+}r_1d\theta_1.
$$
Notice that $\theta_1=\theta_1(\theta,r)$ and $r_1=r_1(\theta,r)$ so that $d\theta_1=\frac{\partial\theta_1}{\partial\theta}d\theta+\frac{\partial\theta_1}{\partial r}dr$. 
Finally, by Green's formula,
$$
\int_{B^+}d\theta dr=\int_{B^+}\det g_1^\prime(\theta,r)d\theta dr
$$
that implies our claim.\\

\textbf{\textit{Step $2$}}. Define $\Omega(\theta,r)=1-\det g_1^\prime (\theta,r)$. Then there exist two $C^1$ functions $\alpha(\theta,r)$ and $\beta(\theta,r)$ $2\pi$-periodic in $\theta$ that vanish on $\partial B$ and such that $\Omega=\frac{\partial \beta}{\partial r}-\frac{\partial \alpha}{\partial\theta}$.\\
 Consider the two functions 
$$
\alpha(\theta,r)=-\int_0^\theta[\Omega(\Theta,r)-\bar{\Omega}(r)]d\Theta \quad\mbox{and}\quad \beta(\theta,r)=\int_{a_0}^r\bar{\Omega}(\rho)d\rho.
$$
with $\bar{\Omega}(r)=\frac{1}{2\pi}\int_0^{2\pi}\Omega(\theta,r)d\theta$. First of all they are of class $C^1$ because $g_1$ is of class $C^2$. \\
Notice that from Step 1 we have that
\begin{equation}\label{un}
\int_{a_0}^b\int_0^{2\pi}\Omega(\theta,r)d\theta dr=0,
\end{equation}
and, remembering that $g_1$ is the identity on a neighbourhood of $\partial B$, it results
\begin{equation}\label{du}
\Omega(\theta,b)=0
\end{equation}
Moreover, an exact symplectic map is also area preserving so the determinant of the Jacobian is $1$. It follows that
\begin{equation}\label{tr}
\Omega(\theta,r)=0 \mbox{   on   } \mathbb{S}^1\times [a_0,a].
\end{equation}
By computation we get $\Omega=\frac{\partial \beta}{\partial r}-\frac{\partial \alpha}{\partial\theta}$ on $B^+$. The fact that $\alpha$ and $\beta$ are $2\pi$-periodic with respect to $\theta$ is trivial for $\beta$, while comes from (\ref{un}) for $\alpha$. Moreover, remembering (\ref{du}) we have
$$
\alpha(\theta,b)=-\int_0^\theta[\Omega(\Theta,b)-\bar{\Omega}(b)]d\Theta=0
$$
and by (\ref{un})
$$
\beta(\theta,b)=\frac{1}{2\pi}\int_{a_0}^b\int_0^{2\pi}\Omega(\theta,r)d\theta=0
$$
and, by (\ref{tr}), $\alpha$ and $\beta$ vanish on $\mathbb{S}^1\times[a_0,a]$. We can do the same on $\mathbb{S}^1\times [-b_0,-a_0]$ and find $\alpha(\theta,r)$ and $\beta(\theta,r)$ with the same property. Finally we can extend these functions to all $B$ setting $\alpha(\theta,r)=0$ and $\beta(\theta,r)=0$ on $A_0$ and the properties on $\mathbb{S}^1\times[a_0,a]$ guarantee the regularity.\\

\textbf{\textit{Step $3$}}.
Consider the function, for $t\in[0,1]$, $\Omega_t(\theta,r)=(1-t)+t\det g_1^\prime(\theta,r)$ and define the vector field 
$$
X_1(t,\theta,r)=\frac{1}{\Omega_t(\theta,r)}\alpha(\theta,r),\quad X_2(t,\theta,r)=-\frac{1}{\Omega_t(\theta,r)}\beta(\theta,r)
$$
and the associated differential equation
$$
\dot{\theta}=X_1(t,\theta,r),\quad \dot{r}=X_2(t,\theta,r)
$$
with solution $\phi_t =(\Theta_t, R_t)$ passing through $(\theta,r)$ at time $t=0$. The solution is unique because $X_1$ and $X_2$ are of class $C^1$ (this justifies the hypothesis of $S$ being $C^2$). We claim that
$$
\Omega_t(\Theta_t,R_t)\det (\frac{\partial(\Theta_t,R_t)}{\partial(\theta,r)})=1, \quad t\in[0,1].
$$
Remember that a map isotopic to the identity is also orientation preserving, while the converse is false in the cylinder (as a counterexample take the map $(\theta,r)\mapsto (-\theta,-r)$). Hence we have that $\det g_1^\prime>0$ so that the vector field is well defined. Notice that if $(\theta,r)\in B$ the solution does not leave $B$ because the boundary circles of $B$ are continua of fixed points: it implies that $\phi_t(B)=B$. 
Using Liouville formula for the linearized equation we have, for every $t\in[0,1]$
\begin{equation*}
\begin{split}
&\Omega_t(\Theta_t,R_t)\det (\frac{\partial(\Theta_t,R_t)}{\partial(\theta,r)})\\
&=\Omega_t(\Theta_t,R_t)\exp\{\int_0^ttr(\frac{\partial X_1}{\partial\theta}+\frac{\partial X_2}{\partial r})(s,\Theta_s,R_s)ds   \}\\
&=\Omega_t(\Theta_t,R_t)\exp\{\int_0^t(-\frac{1}{\Omega^2_t}\alpha\frac{\partial\Omega_t}{\partial \theta}+\frac{1}{\Omega^2_t}\beta\frac{\partial\Omega_t}{\partial r}-\frac{1}{\Omega_t}\frac{\partial\Omega_t}{\partial t} )ds \}
\end{split}
\end{equation*}
where in the last equality we used the properties of $\Omega$ and the fact that $\partial\Omega_t/\partial t=-\Omega$.  

\noindent So, we have to prove that
$$
\Omega_t(\Theta_t,R_t)\exp\{\int_0^t(-\frac{1}{\Omega^2_t}\alpha\frac{\partial\Omega_t}{\partial \theta}+\frac{1}{\Omega^2_t}\beta\frac{\partial\Omega_t}{\partial r}-\frac{1}{\Omega_t}\frac{\partial\Omega_t}{\partial t})ds  \}=1.
$$ 
Passing to the logarithm and differentiating with respect to $t$ this equality is true if and only if
$$
-\frac{1}{\Omega^2_t}\alpha\frac{\partial\Omega_t}{\partial \theta}+\frac{1}{\Omega^2_t}\beta\frac{\partial\Omega_t}{\partial r}-\frac{1}{\Omega_t}\frac{\partial\Omega_t}{\partial t}= -\frac{1}{\Omega_t}[\frac{\partial\Omega_t}{\partial \theta}\dot{\Theta}_t+\frac{\partial\Omega_t}{\partial r}\dot{R}_t+\frac{\partial\Omega_t}{\partial t}]
$$    
that is the case remembering the definition of $\Theta_t$ and $R_t$.\\

\textbf{\textit{Step $4$}}. The function $g_2=g_1\circ\phi_1$ satisfies the lemma.\\
Indeed, by the previous step  
$$
\det g_2^\prime=\det(g_1^\prime\circ\phi_1)\det\phi_1^\prime=\Omega_1(\Theta_1,R_1)\det (\frac{\partial(\Theta_1,R_1)}{\partial(\theta,r)})=1
$$ 
that means that is area preserving. Moreover, by the definition of the vector field $(X_1,X_2)$ we have $\phi_1|_{\partial B}=Id$ and $\phi_1|_{A_0}=Id$ that imply $g_2|_{\partial B}=Id$ and $g_2|_{A_0}=S$.
\end{proof}

With the same notation let us conclude with the proof of Lemma \ref{claim2}. This part does not involve differential forms so we report the version by Franks.

\begin{proof}[Proof of Lemma \ref{claim2}]
Let $\tilde{g}_2:\tilde{B}\rightarrow\tilde{B}$ be the lift (fixed by the boundary twist condition) of $g_2$ that extends $\tilde{f}:\tilde{A}\rightarrow\tilde{B}$. Now consider 
$$
M_0:=\sup_{x\in\tilde{B}}d(\tilde{g}_2(x),x)
$$ 
where $d$ is the distance in $\mathbb{R}^2$ and fix $M>M_0$. So we have that $M$ is greater than the distance that a point in $\tilde{B}$ could be moved by $\tilde{g}_2$. Now consider the strip $\tilde{A}^+=\mathbb{R}\times[a_0,a_0+\epsilon]\subset\tilde{A}$ such that by the boundary twist condition and continuity we have that for all $x\in\tilde{A}^+$
\begin{equation}\label{rew}
P(\tilde{g_2}(x))-P(x)>r
\end{equation}
where $P(x_1,x_2)=x_1$ is the projection on the first component.

\noindent Define $\tilde{h}:\tilde{B}\rightarrow\tilde{B}$ by
$$
\tilde{h}(\theta,r)=(\theta+M\rho(r),r)
$$
where $\rho(r)$ is smooth, monotone such that $\rho(r)=0$ for $r<a_0$ and $\rho(r)=1$ for $r>a_0+\epsilon$. Notice that $\tilde{h}$ is area preserving, it is the identity for $r<a_0$, it is a translation by $M$ if $r>a_0+\epsilon$, if $x\in\tilde{A}^+$ then $\tilde{h}(x)\in\tilde{A}^+$ and $P(\tilde{h}(x))>P(x)$.\\
Finally consider 
$$
\tilde{g}_3=\tilde{g_2}\circ\tilde{h}.
$$
For $x\in\tilde{A}^+$ we have, using (\ref{rew})
$$
P(\tilde{g}_3(x))-P(x)=P(\tilde{g}_2(\tilde{h}(x)))-P(x)>r+P(\tilde{h}(x))-P(x)>r>0
$$
which means that we do not have fixed points in $\tilde{A}^+$. Moreover, if we take $x=(\theta,r)\in\mathbb{R}\times[a_0+\epsilon,b]$ then $\tilde{g}_3(x)=\tilde{g}_2(\tilde{h}(x))=\tilde{g}_2(\theta+M,r)$ and by definition of $M$ that means that we do not have fixed points in $\mathbb{R}\times[a_0+\epsilon,b]$ and the boundary twist condition is satisfied on $\mathbb{R}\times \{b\}$.\\
To conclude we consider $\tilde{A}^-=\mathbb{R}\times[-a_0-\epsilon,-a_0]$ and define analogously $\tilde{h^\prime}(\theta,r)=(\theta-M\rho(r),r)$ with similar properties of $\tilde{h}$. Defining $\tilde{g}=\tilde{g}_3\circ\tilde{h^\prime}$ we get also the complete boundary twist condition.
\end{proof}

Let us conclude with a remark on the stability of such fixed points. Remember that a fixed point $p$ of a one-to-one continuous map $S:U\subset\mathbb{R}^N\rightarrow\mathbb{R}^N$ is said to be stable in the sense of Liapunov if for every neighbourhood $V$ of $p$ there exists another neighbourhood $W\subset V$ such that, for each $n>0$, $S^n(W)$ is well defined and $S^n(W)\subset V$. We have:
\begin{cor}\label{inst}
If $S$ is analytic, at least one of the fixed points coming from theorem \ref{franks} is unstable.
\end{cor}
\begin{proof}
For the special case of dimension two, there exists a relation between the stability of a fixed point and its fixed point index. In fact it was proved in \cite{ortega1} that if a continuous one-to-one map $S$ which is also orientation and area preserving has a stable fixed point $p$ then either $S=Id$ in some neighbourhood of $p$ or there exists a sequence of Jordan curves $\{\Gamma_n\}$ converging to $p$ such that, for each $n$,
$$
\Gamma_n\cap Fix(S)=\emptyset,\quad i(S,\hat{\Gamma_n})=1
$$
where $\hat{\Gamma_n}$ is the bounded component of $\mathbb{R}^2\setminus\Gamma_n$.\\
The set of fixed points can be described by the equation 
$$
(F(\theta, r)-\theta)^2+(G(\theta, r)-r)^2=0
$$
This is an analytic subset of the plane, indeed is the set of the zeros of an analytic function. The local structure of these sets is described in \cite{krants}: they can contain arcs, isolated points and points with a finite number of branches emanating from them. In the cases of a non isolated fixed point we can not find a sequence of Jordan curves converging to the point and not crossing the set, so we have instability. In the case of isolated fixed points, remember that the Euler characteristic of the strip is null, so, using the Poincaré-Hopf index formula for manifolds with boundary \cite[p.447 Theorem 3.1 and p.233 Proposition 4.5]{granasdugundji}, at least one fixed point does not have positive index and so it is unstable.    

\end{proof}

\subsection{The twist condition}\label{sectw}     
In this section we will consider the case in which the map $S$ satisfies a twist condition. More precisely, using the same notation of the previous section, consider a map $S:\Omega\rightarrow\mathbb{R}$ where $\Omega=\{(\theta,r)\in\mathbb{R}^2 : a<r<\psi(\theta)\}$, $a$ is a fixed constant and $\psi:\mathbb{R}\rightarrow]a,+\infty ]$ is a $2\pi$-periodic, lower semi-continuous function.
\begin{defin} 
The defined above map $S$ satisfies the \emph{twist condition} if\newline $\partial\theta_1\ / \partial r>0$. If the map is defined from an annulus into itself, we say that satisfies the \emph{twist condition} if the lift satisfies the same condition.
\end{defin}
\noindent Note that the just introduced \emph{twist condition} and the \emph{boundary twist condition} used in Theorem \ref{franks} and \ref{pbstrip} are independent. Indeed, consider the annulus $A=[0,2\pi]\times[-1,1]$: the map
\begin{equation*}
\begin{sistema}
\theta_1=\theta+e^r  \\
r_1=r
\end{sistema}
\end{equation*}
satisfies the twist condition but not the boundary twist one; the map
\begin{equation*} 
\begin{sistema}
\theta_1=\theta+r^2+\frac{3}{2}r  \\
r_1=r
\end{sistema}
\end{equation*}
satisfies the boundary twist condition but not the twist one. Moreover the two conditions coexist in the map
\begin{equation*}
\begin{sistema}
\theta_1=\theta+r  \\
r_1=r.
\end{sistema}
\end{equation*}
The twist condition is important in itself because it is fundamental for the application of the Aubry-Mather theory. In our case it will lead us to more results on the equation of the forced relativistic pendulum.\\    
It can be proved \cite{ortega3} that
\begin{teo}\label{ortteo}
Assume that $S$ is exact symplectic and satisfies the twist condition. Fix an integer $N$ and assume that for each $\theta\in\mathbb{R}$ there exists $r_\theta\in]a,\psi(\theta)[$ with 
\begin{equation}\label{cond}
F(\theta,a)<\theta+2N\pi<F(\theta,r_\theta).
\end{equation}
Then the system 
\begin{equation}\label{sist}
\begin{sistema}
F(\theta,r)=\theta+2N\pi \\
G(\theta,r)=r,
\end{sistema}
\end{equation}
with $\theta\in[0,2\pi[, (\theta,r)\in\Omega$, has at least two solutions.
\end{teo}
We will explore some consequences of this theorem but let us first point out some remarks on the concept of \emph{intersection property}.\\ 
Consider the cylinder $C=\mathbb{S}^1\times\mathbb{R}$ with the usual covering  map $\Pi(\theta, r)=(\overline{\theta}, r)$; by the periodicity of $S$, we can define a new map, also denoted by $S$, mapping $\Pi(\Omega)$ into $C$. Consider also a non contractible Jordan curve $\Gamma$ in $C$ with positive orientation. The curve $\Gamma$ divides $C$ in two connected components and let us call $R_-(\Gamma)$ the lower one and $R_+(\Gamma)$ the other one; similarly the curve $\Gamma_1=S(\Gamma)$ divides $C$ in two connected components and let us call $R_-(\Gamma_1)$ and $R_+(\Gamma_1)$.  
We can give the following
\begin{defin}
Using the notations given above, we say that a map $S:C\rightarrow C$ has the \emph{intersection property} if for every non-contractible Jordan curve $\Gamma\subset C$ that is the graph of a function,
$$
\Gamma\cap S(\Gamma)\neq\emptyset.
$$
A map $S:C\rightarrow C$ has the \emph{strong intersection property} if it has the intersection property and for every non-contractible Jordan curve $\Gamma\subset C$ that is the graph of a function, either $\Gamma=S(\Gamma)$ or 
$$
S(\Gamma)\cap R_+(\Gamma)\neq\emptyset\quad\mbox{and}\quad S(\Gamma)\cap R_-(\Gamma)\neq\emptyset.
$$
\end{defin}
In the proof of theorem \ref{ortteo} it was proved that if a map $S:C\rightarrow C$ is exact symplectic and preserves the orientation, then it has the strong intersection property. Note that the strong intersection property implies the intersection property, while the converse is false. As a counterexample consider the following map of the cylinder in itself:
\begin{equation*}
\begin{sistema}
\theta_1=\theta  \\
r_1=r+\psi(\theta)
\end{sistema}
\end{equation*}
where $\psi(\theta)$ is a non-negative continuous function such that exists $\theta_*$ such that $\psi(\theta_*)=0$.\\ 
Moreover we have that

\begin{lemma}\label{interse1}
If a map $S:C\rightarrow C$ has the strong intersection property, then for every non-contractible Jordan curve  $\Gamma\subset C$ that is the graph of a function we have
$$
\#\{p\in\Gamma\cap S(\Gamma)\}\geq 2.
$$
\end{lemma}
\begin{proof}
The case $\Gamma=S(\Gamma)$ is trivial. In the other case the strong intersection property implies that there exist two points $p_+\in R_+(\Gamma)$ and $p_-\in R_-(\Gamma)$ that are connected by an arc $\gamma\subset S(\Gamma)$. But $S(\Gamma)$ is a Jordan curve so there must exists another arc $\gamma_*\subset S(\Gamma)$ connecting $p_-$ to $p_+$ and crossing $\Gamma$ in a point $p_*\neq p$. 
\end{proof}

Now we are ready to prove two corollaries of theorem \ref{ortteo}  

\begin{cor}\label{cor1}
If in theorem \ref{ortteo} we require also that $S(\theta,r)$ is analytic then the set of the solutions of system (\ref{sist}) is either finite or the graph of an analytic $2\pi$-periodic function.   
\end{cor}
\begin{proof}
By (\ref{cond}) and the twist condition, we get that for each $\theta$ the equation 
\begin{equation}\label{eqre}
F(\theta,r)=\theta+2N\pi
\end{equation} 
has a unique solution $r:=\phi(\theta)$. By the uniqueness we have that $\phi$ is $2\pi$-periodic. Moreover, because of the twist condition we can apply, for a fixed $\theta$, the analytic version of the implicit function theorem and get an open neighbourhood $U_\theta$ and an analytic function $\tilde{\phi}:U_\theta\rightarrow\mathbb{R}$ such that $F(\theta,\tilde{\phi}(\theta))=\theta+2N\pi$ on $U_\theta$. But, by uniqueness, $\phi(\theta)=\tilde{\phi}(\theta)$ on $U_\theta$. Repeating the argument for each $\theta$, we get that $\phi$ is also analytic. So, $S(\theta,\phi(\theta))$ is the graph of an analytic function in the cylinder: let us call it $\phi_1(\theta)$. This comes from the analyticity and the periodicity of $S$ and the fact that $\phi(\theta)$ satisfies equation (\ref{eqre}). So, by Lemma \ref{interse1}, $\phi$ and $\phi_1$ must intersect in at least two points of the cylinder that are the solutions of system (\ref{sist}) when elevated. Moreover, from the theory of analytic functions, we know that either the set $\{\theta\in[0,2\pi]:\phi(\theta)=\phi_1(\theta)\}$ is finite or $\phi(\theta)=\phi_1(\theta)\quad\forall\theta$.\\
\end{proof}

\begin{cor}\label{cor2}
Suppose that in theorem \ref{ortteo}, condition (\ref{cond}) is satisfied for some $N\in\mathbb{Z}$. Let $(\hat{\theta},\hat{r})$ be an isolated solution of system (\ref{sist}) and define the map $T(\theta,r)=(\theta+2\pi,r)$. Then $i (T^{-N}S,(\hat{\theta},\hat{r}))$ is either $-1$ or $0$ or $1$. 
\end{cor}
\begin{proof}
First of all notice that $i (T^{-N}S,(\hat{\theta},\hat{r}))$ is well defined because $(\hat{\theta},\hat{r})$ is an isolated fixed point of $T^{-N}S$. To compute the index remember that 
$$
i (T^{-N}S,(\hat{\theta},\hat{r}))=\deg (T^{-N}S-Id, B_\delta(\hat{\theta},\hat{r})) 
$$
where $\deg$ indicates the Brouwer degree and $\delta$ could be chosen small enough by the excision property. So we will deal with the degree of the map
\begin{equation*}
\begin{split}
&(T^{-N}S-Id)(\theta,r)=(F(\theta,r)-2N\pi-\theta, G(\theta,r)-r)\\
&:=(\tilde{F}(\theta,r),\tilde{G}(\theta,r)):=\tilde{S}(\theta,r)
\end{split}
\end{equation*}
and to compute it we will use a technique by Krasnosel'skii \cite{kras} that allows to reduce the dimension.\\
By the hypothesis, the point $(\hat{\theta},\hat{r})$ is an isolated zero of $\tilde{F}(\theta,r)$ and by the twist condition $\frac{\partial\tilde{F}}{\partial r}=\frac{\partial F}{\partial r}>0$. So we can apply the implicit function theorem to the equation $\tilde{F}(\theta,r)=0$ and find a $C^1$ function $\phi(\theta)$ defined on a neighbourhood of $\hat{\theta}$ such that $\tilde{F}(\theta,\phi(\theta))=0$ and $\phi(\hat{\theta})=\hat{r}$. Hence it is well defined the function
$$
\Phi(\theta)=\tilde{G}(\theta,\phi(\theta))
$$  
which has $\hat{\theta}$ as an isolated zero.
Now consider the homotopy
\begin{equation*}
H((\theta,r),\lambda)=
\begin{sistema}
\lambda\tilde{F}(\theta,r)+(1-\lambda)(r-\phi(\theta))\\
\lambda\tilde{G}(\theta,r)+(1-\lambda)\Phi(\theta).
\end{sistema}
\end{equation*}
We claim that it is admissible i.e. $(\hat{\theta},\hat{r})$ is an isolated zero for every $\lambda$.
Indeed consider the system
\begin{equation}\label{omo}
\begin{sistema}
\lambda\tilde{F}(\theta,r)+(1-\lambda)(r-\phi(\theta))=0\\
\lambda\tilde{G}(\theta,r)+(1-\lambda)\Phi(\theta)=0.
\end{sistema}
\end{equation}
Because of the twist condition, if we define $\mathcal{F}(\theta,r,\lambda)=\lambda\tilde{F}(\theta,r)+(1-\lambda)(r-\phi(\theta))$, we have
$$
\frac{\partial\mathcal{F}}{\partial r}(\hat{\theta},\hat{r},\lambda)=\lambda\frac{\partial\tilde{F}}{\partial r}(\hat{\theta},\hat{r})+(1-\lambda)>0
$$
and so we can apply the implicit function theorem to solve the first equation in a neighbourhood of $\hat{\theta}$ and by the uniqueness the only solution is $r=\phi(\theta)$. Substituting it in the second equation we get, because of the definition of $\Phi(\theta)$,
$$
\lambda\tilde{G}(\theta,\phi(\theta))+(1-\lambda)\Phi(\theta)=0\Rightarrow \lambda\Phi(\theta)+(1-\lambda)\Phi(\theta)=0 \Rightarrow\Phi(\theta)=0
$$
that, remember, has $\hat{\theta}$ as an isolated solution. So $(\hat{\theta},\hat{r})$ is an isolated solution of system (\ref{omo}) and we can choose $\delta$ small enough such that $(\hat{\theta},\hat{r})$ is the only solution in $B_\delta(\hat{\theta},\hat{r})$.
So we are led to the computation of the degree of the map
$$
(\theta,r)\longmapsto (r-\phi(\theta),\Phi(\theta))
$$  
that, if $\Phi^\prime(\hat{\theta})\neq 0$, can be easily computed by linearization.\\
However it could happen that $\Phi^\prime(\hat{\theta})= 0$. So consider the other homotopy     
\begin{equation*}
H((\theta,r)\lambda)=
\begin{sistema}
\lambda\tilde{F}(\hat{\theta},r)+(1-\lambda)(r-\phi(\theta))\\
\Phi(\theta)
\end{sistema}
\end{equation*}
where $\hat{\theta}$ is fixed.
To prove that it is admissible, consider the system
\begin{equation*}
\begin{sistema}
\lambda\tilde{F}(\hat{\theta},r)+(1-\lambda)(r-\phi(\theta))=0\\
\Phi(\theta)=0.
\end{sistema}
\end{equation*}
By the definition of $\Phi(\theta)$ we have that $\hat{\theta}$ is an isolated solution of the second equation that, substituted in the first one, gives $\lambda\tilde{F}(\hat{\theta},r)+(1-\lambda)(r-\hat{r})=0$. We have that $\hat{r}$ is a solution and is also the only one, because by the twist condition we have
$$
\frac{\partial}{\partial r}[\lambda\tilde{F}(\hat{\theta},r)+(1-\lambda)(r-\hat{r})]>0.
$$  
So $(\hat{\theta},\hat{r})$ is an isolated solution of the system, the homotopy is admissible and we can compute the degree of the function
$$
W(\theta,r)=(\tilde{F}(\hat{\theta},r),\Phi(\theta)).
$$
To use the factorization property of the degree consider the function $L(x,y)=(y,x)$.
We have
$$
\deg(L\circ W, B_\delta(\hat{\theta},\hat{r}))=\deg(L,B_\delta(0,0))\deg(W, B_\delta(\hat{\theta},\hat{r}))=-\deg(W, B_\delta(\hat{\theta},\hat{r})).
$$
Now, by the factorization property
\begin{equation*}
\begin{split}
&\deg(W, B_\delta(\hat{\theta},\hat{r}))=-\deg(\tilde{F},I_{\hat{r}})\deg(\Phi ,I_{\hat{\theta}})\\=&-sign\{\frac{\partial F}{\partial  r}(\hat{\theta},\hat{r})\}\deg(\Phi ,I_{\hat{\theta}})=-\deg(\Phi ,I_{\hat{\theta}}).
\end{split}
\end{equation*}
The function $\Phi$ is defined in dimension 1 so its degree can be either $0$ or $1$ or $-1$.
Finally $i(T^{-N}S,(\hat{\theta},\hat{r}))$ can be either $0$ or $1$ or $-1$.   
\end{proof}

\begin{remark}\label{dise}
An intuitive idea of when these cases could occur is given 
by figure \ref{figint}. 

%
%

\begin{figure}\label{figint}
\centering
\subfigure[$i=+1$, $\Phi^\prime(\hat{\theta})\neq 0$]
{\includegraphics[width=6cm]{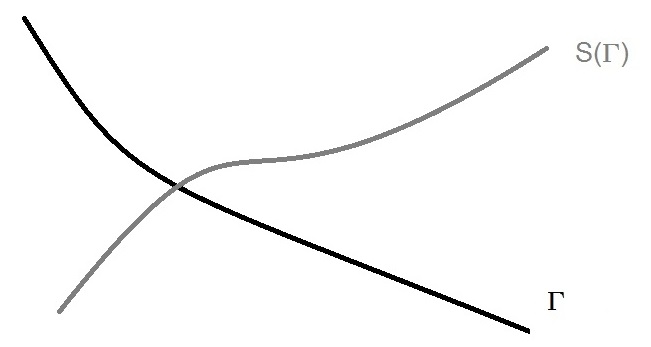}}
\hspace{5mm}
\subfigure[$i=+1$, $\Phi^\prime(\hat{\theta})= 0$]
{\includegraphics[width=6cm]{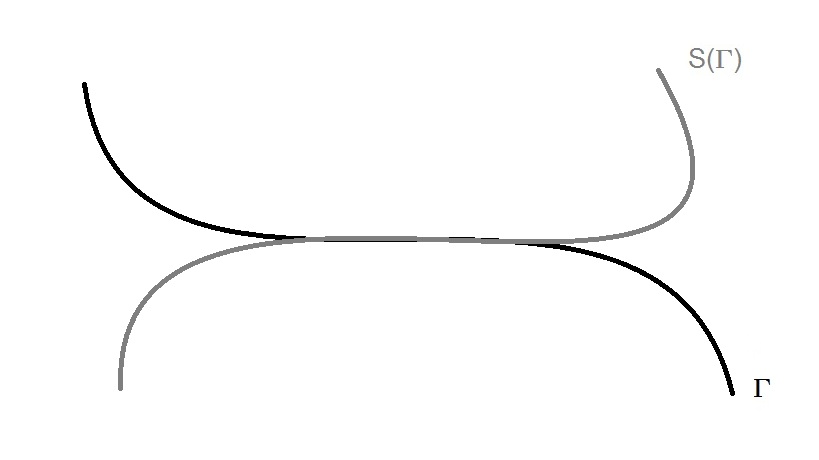}}
\hspace{5mm}
\subfigure[$i=0$, $\Phi^\prime(\hat{\theta})= 0$]
{\includegraphics[width=6cm]{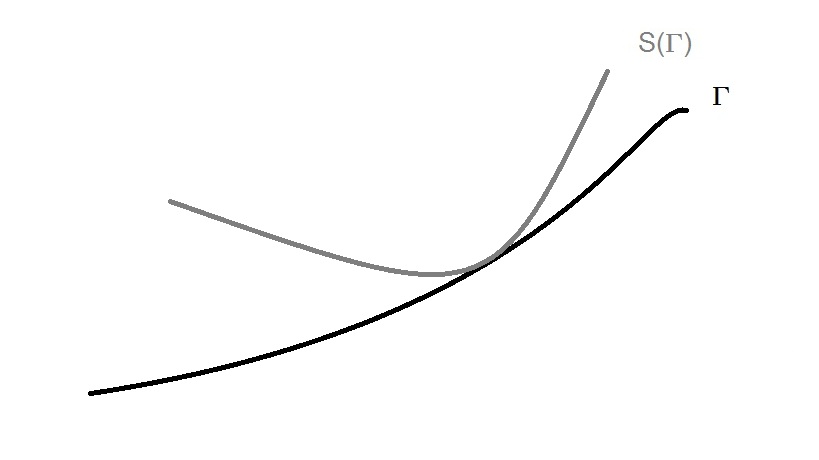}}
\hspace{5mm}
\subfigure[$i=-1$, $\Phi^\prime(\hat{\theta})= 0$]
{\includegraphics[width=6cm]{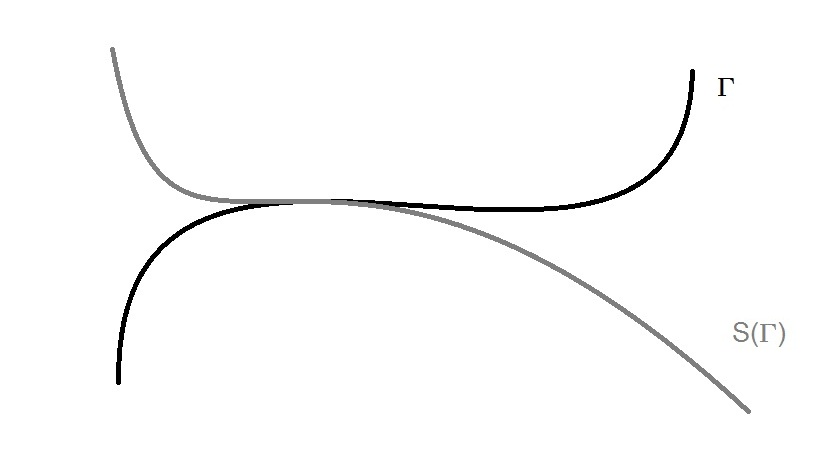}}
\hspace{5mm}
\subfigure[$i=-1$, $\Phi^\prime(\hat{\theta})\neq 0$]
{\includegraphics[width=6cm]{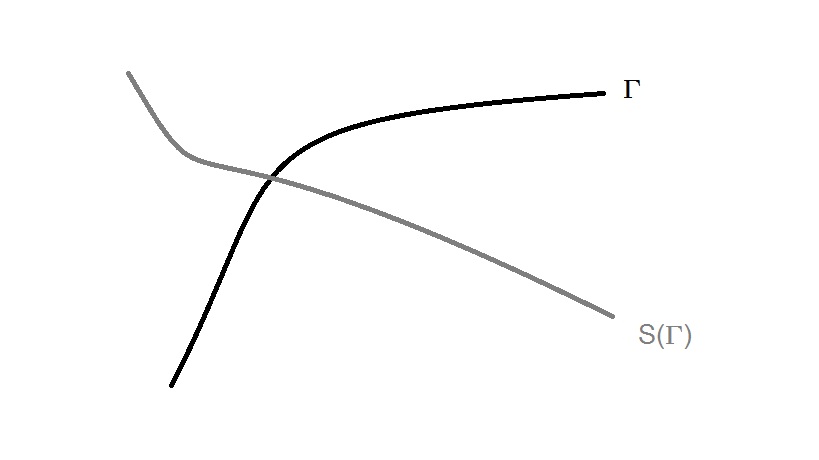}}
\caption{Possibilities of intersections}
\end{figure}
\end{remark}

\newpage

\section{Applications to the Poincaré map induced by a forced relativistic pendulum}

\subsection{Proof of theorem \ref{introma}}
We are looking for $T$-periodic solutions with winding number $N$ of the equation  

\begin{equation}\label{rpend}
\frac{d}{dt}\left(\frac{\dot{x} }{ \sqrt{ 1- \dot{x}^2 } }\right) + a\sin x=f(t),
\end{equation}
under the hypothesis previously stated in the introduction.\\
First of all notice that physical intuition suggests that it should not be possible to have such solutions for every $N$ and $T$, because of the bound given by the speed of light. This is a necessary condition that holds for a larger class of equations, namely: 
\begin{prop}\label{neces}
Let $x(t)$ be a $T$-periodic solution with winding number $N$ of
\begin{equation}\label{pengengen}
\frac{d}{dt}\left(\frac{\dot{x} }{ \sqrt{ 1- \dot{x}^2 } }\right) =F(t,x)
\end{equation}
where $F(t,x)$ is continuous and $T$-periodic in $t$.\\ 
Then 
\begin{equation}\label{condnec}
\left|\frac{2N\pi}{T}\right|<1.
\end{equation}
\end{prop}
\begin{proof}
By Lagrange theorem we get
\begin{equation*}
\begin{split}
|2N \pi | = | x(t+T)-x(t)| = |\dot{x}(c)T|  
\end{split}
\end{equation*} 
for some $c\in(t,t+T)$.
But the domain of equation (\ref{pengengen}) is $|\dot{x}(t)|<1$ for all $t$, so
$$
|2N\pi|<T.
$$
\end{proof}

\noindent In this section we will see why the relativistic condition (\ref{condnec}) is also sufficient to have $T$-periodic solutions with winding number $N$. The proof will be an application of theorem \ref{franks} considering $\tilde{S}$ the Poincaré map.\\

\noindent First of all, let us perform the change of variables
\begin{equation}\label{cvper}
y(t)=x(t)-\frac{2N\pi}{T}t.
\end{equation}
Notice that in this way $y(t+T)=y(t)$ and $T$-periodic solutions with winding number $N$ of (\ref{rpend}) correspond to classical $T$-periodic solution of
\begin{equation}\label{pendgen2}
\frac{d}{dt}\left(\frac{\dot{y}+\frac{2N\pi}{T} }{ \sqrt{ 1- (\dot{y}+\frac{2N\pi}{T})^2 } }\right) + a\sin \left(y+\frac{2N\pi}{T}t\right)=f(t).
\end{equation}
     
\noindent We will find $T$-periodic solutions of equation (\ref{pendgen2}) as fixed points of the Poincaré map.\\
Equation (\ref{pendgen2}) can be seen as the Euler-Lagrange equation coming from the Lagrangian
$$
L(y,\dot{y},t)=-\sqrt{ 1- (\dot{y}+\frac{2N\pi}{T})^2 }+a\cos \left(y+\frac{2N\pi}{T}t\right)+f(t)y
$$
and presents singularities.
On the other hand, if we perform the change of variables given by the Legendre Transform
\begin{equation}\label{cv}
\begin{sistema}
q=y\\
p=\frac{\partial L}{\partial\dot{y}}=\frac{\dot{y}+\frac{2N\pi}{T}}{\sqrt{1-(\dot{y}+\frac{2N\pi}{T})^2}},
\end{sistema}
\end{equation}
we get the Hamiltonian 
$$H(q,p,t)=[p\dot{y}-L(q,\dot{x},t)]_{\dot{y}=\dot{y}(p)}=\sqrt{p^2+1}-\frac{2N\pi}{T}p-a\cos (q+\frac{2N\pi}{T}t)-f(t)q$$
and the new Hamiltonian system
\begin{equation}\label{rpend2}
\begin{sistema}
\dot{q}=\frac{p}{\sqrt{1+p^2}}-\frac{2N\pi}{T}\\
\dot{p}=-a\sin \left(q+\frac{2N\pi}{T}t\right)+f(t).
\end{sistema}
\end{equation} 
that has no singularities and has solutions globally defined.\\
\smallskip
\noindent Now let $S(q_0,p_0)=(Q(q_0,p_0),P(q_0,p_0))=(q(T,q_0,p_0),p(T,q_0,p_0))$ be the Poincaré map associated to system (\ref{rpend2}) that is well defined because of the boundedness of the second term. First of all notice that the Poincaré map is isotopic to the inclusion. The isotopy is simply given by the flow $\phi$, in fact we have the isotopy
$$
\phi(tT,q_0,p_0)\quad t\in[0,1].
$$
Notice that this isotopy is valid also in the cylinder because by periodicity we have $$\phi(t,q_0+2\pi,p_0)=\phi(t,q_0,p_0)+(2\pi,0).$$ 
Moreover, inspired by \cite{ortegakunze} we can prove
\begin{lemma}\label{l6}
The Poincaré map $S$ is exact symplectic in the cylinder.
\end{lemma}
\begin{proof}
Let us call $\theta=q_0, r=p_0$ and $K=\frac{2N\pi}{T}$.
Consider the function
\begin{equation*}
\begin{split}
& V(\theta,r)=\int_0^T [p(\frac{p}{\sqrt{1+p^2}}-K)-H(q,p,t)]dt   \\
&=\int_0^{T}[-\sqrt{\frac{1}{1+p^2(t,\theta,r)}}+a\cos (q(t,\theta,r)+Kt)+f(t)q(t,\theta,r)] dt.
\end{split}
\end{equation*}
First of all, it follows from the periodicity of (\ref{pendgen2}) and the change of variables (\ref{cv}) that $q(t,\theta+2\pi, r)=q(t,\theta, r)+2\pi$ and $p(t,\theta+2\pi,r)=p(t,\theta,r)$. Hence we have
\begin{equation*}
V(\theta+2\pi,r)=V(\theta, r)
\end{equation*}
using the hypothesis of the null mean value of $f$.
Now let us compute the differential $dV$.
We have
\begin{equation}\label{htf}
\begin{split}
V_\theta= & \int_0^T[\frac{p}{(1+p^2)^{3/2}} \frac{\partial p}{\partial\theta}+(-a\sin (q+Kt)+f(t))\frac{\partial q}{\partial\theta}] dt \\
   &=\int_0^T[\frac{p}{(1+p^2)^{3/2}} \frac{\partial p}{\partial\theta}+\dot{p}\frac{\partial q}{\partial\theta}] dt
\end{split}
\end{equation}
using the second equation in (\ref{rpend2}).
Now, integrating by parts and using the first equation in (\ref{rpend2}) we get
\begin{equation*}
\int_0^{T}\dot{p} \frac{\partial q}{\partial\theta}dt=[p\frac{\partial q}{\partial\theta}]_0^{T}-\int_0^{T} p \frac{\partial \dot{q}}{\partial\theta} dt=[p\frac{\partial q}{\partial\theta}]_0^{T}-\int_0^T \frac{p}{(1+p^2)^{3/2}} \frac{\partial p}{\partial\theta}
\end{equation*}
that, substituted in (\ref{htf}) gives
$$
V_\theta=p(T)\frac{\partial q}{\partial\theta}(T)- p(0)\frac{\partial q}{\partial\theta}(0).
$$
Analogously we can get 
$$
V_r=p(T)\frac{\partial q}{\partial r}(T)- p(0)\frac{\partial q}{\partial r}(0).
$$  
Hence $dV=p_1dq_1-pdq$, that means that the function $V$ will satisfy the thesis.
\end{proof}
Finally, the relativistic effect will give also the boundary twist condition:

\begin{lemma}\label{l4}
If $|\frac{2N\pi}{T}|<1$ then there exists $\tilde{p}>0$ and $r>0$ such that 
$$
Q(q,-\tilde{p})-q<-r\quad \mbox{and} \quad Q(q,\tilde{p})-q>r.
$$
\end{lemma}
\begin{proof}
Let us prove the first inequality, being the second similar.
Let us call $K:=\frac{2N\pi}{T},\: |K|<1$ and consider the function, coming from system (\ref{rpend2}),
$$
A(p)=\frac{p}{\sqrt{p^2+1}}.
$$
We have that $A(p)$ is an odd increasing function such that $A(0)=0$ and $\lim_{p\to\pm\infty}A(p)=\pm 1$.
Since $|K|<1$, by continuity, we can find $\hat{p}>0$ such that
\begin{equation*}
\begin{sistema}
A(p)>K\quad\mbox{for}\:p>\hat{p} \\
A(p)<K\quad\mbox{for}\:p<-\hat{p}.
\end{sistema}
\end{equation*}
Now, integrating the second equation of (\ref{rpend2}) we get, for $t\in[0,T]$
$$
p(t)=p_0-\int_0^t a \sin (q(s)+K)ds + \int_0^t f(s)ds\leq p_0 +t(a+\|f\|_\infty)
$$
and so we can find $\tilde{p}>0$ large enough so that if $p_0 < -\tilde{p}$ then $p(t)<-\hat{p}$ for $t\in [0,T].$
It means
$$
\dot{q}(t)=\frac{p(t)}{\sqrt{1+p^2(t)}}-\frac{2N\pi}{T}<0\quad t\in[0,T]
$$
that is $q(t)$ is decreasing if $t\in [0,T]$ so,
$$Q(q_0,-\tilde{p})=q(T,q_0,-\tilde{p})<q(0,q_0,-\tilde{p})=q_0.$$
Now a standard compactness argument concludes the proof.
\end{proof}

\noindent Now it is straightforward the application of theorem \ref{franks} choosing the strip $\tilde{A}=\mathbb{R}\times[-\tilde{p},\tilde{p}]$ and the fact that solutions of system (\ref{rpend2}) are globally defined implies that we can find a larger strip $\tilde{B}$ such that $S(\tilde{A})\subset int \tilde{B}$. Finally we have that the right-hand side of (\ref{rpend2}) is analytic in $(q,p)$ then, by analytic dependence on initial conditions, also the Poincaré map is analytic. Notice that we do not need the analyticity of $f$ \cite[p.44]{lef}. So, using corollary \ref{inst} we get the instability of one solution. \\
Then theorem \ref{introma} is proved.

\begin{remark}
Similar results on the classical pendulum have been obtained by Franks in \cite[Proposition 5.1]{franks}. He proved the existence of fixed points for the Poincaré map using his version of the Poincaré-Birkhoff theorem and affirmed that they should have positive or negative index. This result needs some clarification. In fact there is another possibility: there could be only a continuum of fixed points and the fixed point index could not be defined. Consider the equation of the classical pendulum: the existence or not of forcing terms $f$ of null mean value such that the periodic solutions are represented only by a continuum in still an open problem. Anyway, as a related example consider the equation 
$$
\ddot{y}+a\sin(y +\frac{2\pi}{T}t)=0
$$
where the potential depends on time.    
Its $T$-periodic solutions correspond, via the change of variables $x=y+\frac{2\pi}{T}t$ to solutions $x(t)$ of
$$
\ddot{x}+a\sin x=0
$$
such that $x(t+T)=x(t)+2\pi$. These solutions forms the graph of a function in the phase space, so it is impossible to define the index.
\end{remark}

\subsection{Proof of theorem \ref{introtw} and consequences}
We will prove the result for the Poincaré map of system (\ref{rpend2}). In the case $a=\pi^2/T^2$ we will have to prevent the function $f$ from being the trigonometric function $a\sin (\frac{2N\pi}{T}t)$.\\
\smallskip 
\noindent According to the previous section let us call $S(q_0,p_0)=(Q(q_0,p_0),P(q_0,p_0))=(q(T,q_0,p_0),p(T,q_0,p_0))$ the Poincaré map associated to system (\ref{rpend2}) in which we suppose $a\leq\pi^2/T^2$. 
We have to prove that $\frac{\partial q}{\partial p_0}(T,q_0,p_0)>0$.\\
If we call
$$
x(t)=\frac{\partial q}{\partial p_0}(t,q_0,p_0)\quad y(t)=\frac{\partial p}{\partial p_0}(t,q_0,p_0)
$$
we know from the elementary theory of ODEs that the vector $(x(t),y(t))$ satisfies the variational equation
$$
\begin{sistema}
\dot{x}=\frac{1}{(1+p^2(t,q_0,p_0))^{3/2}}y\\
\dot{y}=-a\cos (q(t,q_0,p_0)+\frac{2N\pi}{T}t)x\\
x(0)=0\\
y(0)=1
\end{sistema}
$$
that is equivalent to the problem
\begin{equation}\label{sturm11}
\begin{sistema}
\frac{d}{dt}(\dot{x}(1+p^2(t,q_0,p_0))^{3/2}) +a\cos (q(t,q_0,p_0)+\frac{2N\pi}{T}t)x=0\\
x(0)=0\\
\dot{x}(0)=(\frac{1}{p_0^2+1})^{3/2}.
\end{sistema}
\end{equation}
Now consider the equation
\begin{equation}\label{sturm21}
\ddot{z}+\frac{\pi^2}{T^2}z=0
\end{equation}
and first suppose that $a<\pi^2/T^2$. In this case we have that 	
\begin{equation*}
(1+p(t)^2)^{3/2}\geq 1\quad \mbox{and}\quad a\cos (q+\frac{2N\pi}{T}t)\leq a<\pi^2/T^2
\end{equation*}
then (\ref{sturm21}) is a strict Sturm majorant of (\ref{sturm11}). So the Sturm theory and the fact that the function $z(t)=\sin(t\frac{\pi}{T})$ is a solution of (\ref{sturm21}), prove that $x(T)>0$ and the thesis will follow.\\
Now consider the case $a=\pi^2/T^2$. 
First of all the hypothesis $f(t)\neq a\sin(\frac{2N\pi}{T}t)$ prevents $q=2k\pi$ from being a solution. This means that there exists an open subset of positive measure of $[0,T]$ on which $q\neq 2k\pi$ and so
$$
\int_0^{T}\frac{\pi^2}{T^2}\cos q(t)dt < \int_0^{T}\frac{\pi^2}{T^2}dt.
$$ 
In this case we can use a generalization of the classical Sturm separation theorem. It can be achieved adapting the classical proof (cf. \cite{hartman}) to our framework. Consider the argumentum $\theta_1$ and $\theta_2$ respectively of (\ref{sturm11}) and (\ref{sturm21}) coming from the Prufer change of variables; then  we can conclude that $\theta_1(T)>\theta_2(T)$. Remembering that in this framework we have that $x(\tilde{t})=0\Leftrightarrow \theta(\tilde{t})=k\pi$ for some $k\in\mathbb{Z}$ and we are rotating in the clockwise sense, we can conclude using the same argumentation of the previous case translated into the phase-space $(x,p\dot{x})$.    
The proof of theorem \ref{introtw} is complete.\\

\begin{remark}
The condition $a\leq \frac{\pi^2}{T^2}$ is optimal. Indeed suppose $a>\frac{\pi^2}{T^2}$ and consider the autonomous system 
\begin{equation*}
\begin{sistema}
\dot{q}=\frac{p}{\sqrt{1+p^2}}-\frac{2N\pi}{T}\\
\dot{p}=-a\sin (q+\frac{2N\pi}{T}t).\\
\end{sistema}
\end{equation*} 
Notice that $(p=0,q=-\frac{2N\pi}{T}t)$ is an obvious solution. As before consider the variational equation
$$ 
\frac{d}{dt}(\dot{x}(1+p^2(t,q_0,p_0))^{3/2}) +a\cos (q(t,q_0,p_0)+\frac{2N\pi}{T}t)x=0.
$$
Notice that evaluated in the above solution it is nothing but
$$
\ddot{x}+ax=0.
$$
Using Sturm comparison with $\ddot{y}+\frac{\pi^2}{T^2}y=0$ we can conclude analogously as before that $x(T)<0$: it means that we do not have the twist condition. Finally note that in the case $a=\frac{\pi^2}{T^2}$ we have $x(T)=0$ and again the twist condition fails. 
\end{remark}     

\noindent The fact that the Poincaré map satisfies the twist condition allows us to enter in the huge chapter of twist maps. In particular we will get some more results on equation (\ref{introeq}). To state it remember that a $T$-periodic solution with winding number $N$ is said to be \emph{isolated} if there exists $\delta>0$ such that every solution $(q(t),p(t))$ satisfying
\begin{equation*}
0<|q(0)-\hat{q}(0)|+|p(0)-\hat{p}(0)|<\delta
\end{equation*}
is not $T$-periodic with winding number $N$.\\
We have
\begin{teo}\label{gperteo}
If $0<a<\pi^2/T^2$ either the number of isolated $T$-periodic solutions with winding number $N$ is finite or we are in the degenerate case and every degenerate solution is unstable. Moreover, in the first situation, the index of such solution is either $-1$ or $0$ or $1$.  
If we consider the case $a=\pi^2/T^2$ and we add the hypothesis that $f(t)$ is not the trigonometric function $a\sin(\frac{2N\pi}{T}t)$, we get the same results.
\end{teo}
\begin{proof}
By the twist condition we can apply the results of section \ref{sectw}. In particular theorem \ref{ortteo} runs with $\Omega=\{(q,p)\in\mathbb{R}^2: -\tilde{p}<p<\tilde{p}\}$ where $\tilde{p}$ comes from Lemma \ref{l4}. Indeed, if we take $r_\theta=\tilde{p}-\epsilon$ with $\epsilon$ sufficiently small, condition (\ref{cond}) holds with $N=0$ by continuous dependence, and from the previous section we have that the Poincaré map is exact symplectic. This is another way to find two periodic solutions. Notice that it is a weaker result because we need the restriction on the parameter $a$.

\noindent Anyway the Poincaré map is analytic, so, by corollary \ref{cor1}, we have that fixed points either are isolated or form the graph of an analytic $2\pi$-periodic function. Moreover by corollary \ref{cor2} we have the informations on the degree.

\noindent The translation of these results from the Poincaré map to the differential equation gives informations on the periodic solutions of system (\ref{rpend2}) and, by the change of variables (\ref{cvper}) we get analogous results on the $T$-periodic solutions with winding number $N$ of system (\ref{rpend}).
\end{proof}

\section{The autonomous case}
Finally, consider the case $f=0$, i.e. the autonomous equation
\begin{equation}\label{pendau}
\frac{d}{dt}\left(\frac{\dot{x} }{ \sqrt{ 1- \dot{x}^2 } }\right) + a\sin x=0
\end{equation}
with $a\leq \frac{\pi^2}{T^2}$, that can be treated with a phase portrait analysis.
Let us consider the case $ T=2\pi $, so that $ a \leq 1/4$.\\
\smallskip
\noindent First of all it is easily seen that the points $(k\pi,0)$, $k\in\mathbb{Z}$ are constant solutions in the phase space $(x,\dot{x})$.
This analysis is quite simple because the energy
\begin{equation}\label{enau}
E(x,\dot{x})=\frac{1}{\sqrt{1-\dot{x}^2}}-a\cos x+a
\end{equation}
is a first integral and we suddenly reach the conditions 
\begin{equation}\label{condau}
E\geq 1\quad \mbox{and}\quad -1<\dot{x}<1.
\end{equation} 
Remembering that $a\leq 1/4$ we get the phase portrait in figure \ref{ph} where we have the constant solution $(0,0)$ for $E=1$, periodic orbits for $1<E<1+2a$, the heteroclinic orbits for $E=1+2a$ and the unbounded solutions for $E>1+2a$. Moreover, from the first integral (\ref{enau}) we can see the velocity as a function of the time and energy and $\dot{x}\to\pm 1$ as $E\to+\infty$ depending on the sign of $\dot{x}(0)$.

\begin{figure}[h]\label{ph}
\centering
\includegraphics[scale=0.5]{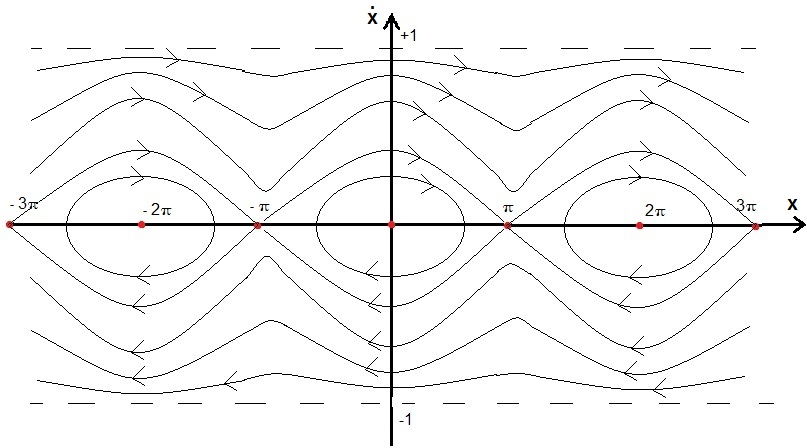}

 \caption{Phase portrait}
 \end{figure}

\noindent Now we turn to the study of the period of the periodic orbits, in particular, in order to complete theorem \ref{introma} we will study the number of $2\pi$-periodic orbits.\\
\begin{prop}
The only $2\pi$-periodic solutions of equation (\ref{pendau}) are the constant ones $(k\pi,0)$ with $k\in\mathbb{Z}$.
\end{prop}
\begin{proof}
We will prove that the period of every orbit (except for the constants one) is strictly greater than $2\pi$. To do so, by the symmetries of the phase portrait, it is enough to prove that for every non-constant periodic orbit, $\dot{x}(\pi,x_0,0)>0$ with $x_0<0$.\\
Let us write the solution $x(t,E)$ such that $x(0,E)=\arccos (1-E/a)$ and $\dot{x}(0,E)=0$ and compute $\frac{\partial  \dot{x}}{\partial E}(\pi,E)$. Remembering (\ref{condau}) and $a\leq1/4$ we have
$$
\frac{\partial \dot{x}}{\partial E}>0\quad\mbox{for}\quad E>1.
$$
Notice that the point $(0,0)$ is a strict minimum of $E(x_0,\dot{x}_0)$ an so
$$
\dot{x}(\pi,E(0,0))<\dot{x}(\pi,E(x_0,\dot{x}_0))\quad\forall(x_0,\dot{x}_0)\neq(0,0).
$$ 
Now, remembering that $E$ is constant on the solutions, we have that for every initial condition $(x_0,\dot{x}_0)$ such that $1<E(x_0,\dot{x}_0)<1+2a$ there exists $\hat{x}<0$ such that $E(x_0,\dot{x}_0)=E(\hat{x},0)$ and $\dot{x}(\pi,E(\hat{x},0))>\dot{x}(\pi,E(0,0))=0$.

\end{proof}

\noindent Looking for $T$-periodic solutions with winding number $N$ we can do the following. By the phase portrait analysis we got that for $E>1+2a$ the solution is unbounded and the orbit in the phase plane is the graph of a function. In this case we will show
\begin{prop}
Fix $|N|\geq 1$ and $T\geq 2\pi$ such that $\frac{2N\pi}{T}<1$. Then there exists exactly one value of the energy $E>1+2a$ such that
$$
x(T+t,E)=x(t,E)+2N\pi.
$$ 
\end{prop}
\begin{proof}
Let us prove in the case $t=0, x(0)=0$ and $N>0$.\\
Remembering the energy (\ref{enau}) we can define a function $T_N(E)$ such that $x(T_N(E))=2N\pi$ (i.e. $T_N(E)$ is the time needed by a solution starting from $0$ at $t=0$ to reach $2N\pi$), namely
\begin{equation*}
T_N(E)=\int_0^{2N\pi}\frac{dx}{\sqrt{1-\frac{1}{(E+a\cos x-a)^2}}}.
\end{equation*} 
Notice that it is continuous, monotone decreasing in $E$ and 
$$
\lim_{E\to 1+2a}T_N(E)=+\infty,\quad \lim_{E\to +\infty}T_N(E)=2N\pi. 
$$ 
The proposition is proved if we can find $E>1+2a$ such that $T_N(E)=T$. It is automatic using the properties just mentioned and the fact that by hypothesis $T>2N\pi$.    

\end{proof}

\bibliographystyle{plain}
\bibliography{biblio4}

\end{document}